\documentclass[a4paper,reqno,11pt]{amsart}

\usepackage[T1]{fontenc}
\usepackage{lmodern}
\usepackage{amssymb,amsmath}
\usepackage{mathtools}
\usepackage{xcolor}
\usepackage[utf8]{inputenc}
\usepackage[unicode=true]{hyperref}
\hypersetup{
  breaklinks=true,
  pdfauthor={Anders Claesson, Atli Fannar Frankl\'{\i}n, Einar Steingr\'{\i}msson},
  pdftitle={Permutations with few inversions},
  colorlinks=true,
  citecolor=blue,
  urlcolor=blue,
  linkcolor=violet,
  pdfborder={0 0 0}}
\usepackage{tikz}
\usetikzlibrary{matrix,arrows,patterns,calc}

\newtheorem{theorem}{Theorem}
\newtheorem{lemma}[theorem]{Lemma}
\newtheorem{proposition}[theorem]{Proposition}
\newtheorem{corollary}[theorem]{Corollary}

\theoremstyle{definition}
\newtheorem{definition}[theorem]{Definition}

\newcommand{\xperm}{\mathcal{C}}
\newcommand{\sym}{\mathcal{S}}
\newcommand{\eps}{\epsilon}
\newcommand{\PIC}{\mathcal{R}}
\newcommand{\Par}{\mathrm{Par}}
\newcommand{\Comp}{\mathrm{Comp}}
\newcommand{\Pic}{R}
\newcommand{\M}{\mathcal{M}}
\newcommand{\Sym}{\mathrm{Sym}}
\newcommand{\sumdiv}{\sigma}

\DeclareMathOperator{\inv}{\mathrm{inv}}
\DeclareMathOperator{\comp}{\mathrm{comp}}

\DeclareMathOperator{\lir}{\mathrm{lir}}
\DeclareMathOperator{\Fix}{\mathrm{Fix}}

\DeclareMathOperator{\cut}{\mathrm{split}}
\DeclareMathOperator{\dmax}{\mathrm{dmax}}

\def\op{\oplus}

\newcommand\abs[1]{\left|#1\right|}

\newcommand{\gr}[1]{{\color{gray}#1}}

\title{Permutations with few inversions}
\author[A.\ Claesson]{Anders Claesson}
\address{Department of Mathematics, University of Iceland, Reykjavik, Iceland}
\email{akc@hi.is}
\author[A.\ F.\ Frankl\'{\i}n]{Atli Fannar Frankl\'{\i}n}
\address{Department of Mathematics, University of Iceland, Reykjavik, Iceland}
\email{aff6@hi.is}
\author[E.\ Steingr\'{\i}msson]{Einar Steingr\'{\i}msson}
\address{Department of Mathematics and Statistics, University of Strathclyde, Glasgow, UK}
\email{einar@alum.mit.edu}
\date{\today}
\begin{document}

\begin{abstract}
  A curious generating function $S_0(x)$ for permutations
  of $[n]$ with exactly $n$ inversions is presented. Moreover,
  $(xC(x))^iS_0(x)$ is shown to be the generating function for
  permutations of $[n]$ with exactly $n-i$ inversions, where $C(x)$ is
  the generating function for the Catalan numbers.
\end{abstract}

\maketitle
\thispagestyle{empty}

\section{Introduction}

The famous triangle of Mahonian numbers starts as follows:\medskip
\[
\begin{array}{rrrrrrrrrrr}
1 & \gr{0} & \gr{0} & \gr{0} & \gr{0} & \gr{0} & \gr{0} & \gr{0} & \gr{0} & \gr{0} & \gr{\cdots} \\
1 & 0 & \gr{0} & \gr{0} & \gr{0} & \gr{0} & \gr{0} & \gr{0} & \gr{0} & \gr{0} & \gr{\cdots} \\
1 & 1 & 0 & \gr{0} & \gr{0} & \gr{0} & \gr{0} & \gr{0} & \gr{0} & \gr{0} & \gr{\cdots} \\
1 & 2 & 2 & 1 & \gr{0} & \gr{0} & \gr{0} & \gr{0} & \gr{0} & \gr{0} & \gr{\cdots} \\
1 & 3 & 5 & 6 & 5 & \gr{3} & \gr{1} & \gr{0} & \gr{0} & \gr{0} & \gr{\cdots} \\
1 & 4 & 9 & 15 & 20 & 22 & \gr{20} & \gr{15} & \gr{9} & \gr{4} & \gr{\cdots} \\
1 & 5 & 14 & 29 & 49 & 71 & 90 & \gr{101} & \gr{101} & \gr{90} & \gr{\cdots} \\
1 & 6 & 20 & 49 & 98 & 169 & 259 & 359 & \gr{455} & \gr{531} & \gr{\cdots} \\
1 & 7 & 27 & 76 & 174 & 343 & 602 & 961 & 1415 & \gr{1940} & \gr{\cdots} \\
1 & \phantom{000}8 & \phantom{00}35 & \phantom{0}111 & \phantom{0}285 & \phantom{0}628 & 1230 & 2191 & 3606 & 5545 & \phantom{00}\gr{\cdots} \\
\end{array}\smallskip
\]
Its $n$-th row records the distribution of inversions on permutations of
$[n]:=\{1,2,\ldots,n\}$. The corresponding generating function is
\begin{equation}\label{q-factorial}
  (1+x)(1+x+x^2)\cdots (1+x+\cdots+x^{n-1}) = \prod_{j = 1}^n \frac{1-x^j}{1-x}.
\end{equation}
We shall derive generating functions for the subdiagonals on or below the
main diagonal of the table above. The first three of those are
\begin{align*}
  S_0(x) &\,=\, 1 + x^3 + 5x^4 + 22x^5 + 90x^6 + 359x^7 + 1415x^8 + \cdots \\
  S_1(x) &\,=\, x + x^2 + 2x^3 + 6x^4 + 20x^5 + 71x^6 + 259x^7 +  961x^8 + \cdots \\
  S_2(x) &\,=\, x^2 + 2x^3 + 5x^4 + 15x^5 + 49x^6 + 169x^7 + 602x^8 + \cdots
\end{align*}
In general, if $i$ is a non-negative integer, then $S_i(x)$ is the
generating function for permutations of $[n]$ with exactly $n-i$
inversions. In other words, if we let $I_n(k)$ denote the number of
permutations of $[n]$ with $k$ inversions, then
\[
  S_i(x)=\sum_{n\geq 0}I_n(n-i)x^n.
\]
It should be noted that there is a known closed expression
for $I_n(k)$ when $k\leq n$, namely the Knuth-Netto formula \cite{taocp-3, netto}:
\begin{multline*}
  \quad I_{n}(k) = \binom{n + k - 1}{k} + \sum_{j = 1}^\infty (-1)^j
    \binom{n + k - u_j - j - 1}{k - u_j - j} \\
    + \sum_{j = 1}^\infty(-1)^j \binom{n + k - u_j - 1}{k - u_j}\quad
\end{multline*}
where $u_j=j(3j-1)/2$ is the $j$-th pentagonal number. This formula can
be proved using \eqref{q-factorial} and Euler's pentagonal number
theorem~\cite{andrews-1983}. For instance, $u_1=1$, $u_2=5$, and the coefficient of $x^6$ in $S_0(x)$ is
\begin{equation*}
  I_6(6) = \binom{11}{6}
    - \binom{10-u_1}{5-u_1}
    - \binom{11-u_1}{6-u_1}
    + \binom{11-u_2}{6-u_2} = 90.
\end{equation*}

Let $C(x) = (1-\sqrt{1-4x})/(2x)$ be the generating function for the
Catalan numbers, $C_n=\binom{2n}{n}/(n+1)$. We show
(Theorem~\ref{Si-S0}) that, for any non-negative integer $i$,
\[
  S_i(x) = \bigl(xC(x)\bigr)^iS_0(x),
\]
thus reducing the problem of determining $S_i(x)$ to that of
determining $S_0(x)$.

Denote by $\sumdiv(n)$ the sum of divisors of $n$, and denote by $p(n)$
the number of integer partitions of $n$. We show (Theorem~\ref{S0-thm}) that
\[
  S_0(x) = R\bigl(xC(x)\bigr),
\]
where the power series $R(x)$ can be expressed in any of the following
three equivalent ways
\begin{align*}
        R(x) &\,=\, \frac{1-x}{1-2x}\prod_{k\geq 1}(1-x^k); \\
  \log R(x)  &\,=\, \sum_{n\geq 1} (2^n-\sumdiv(n)-1) \frac{x^n}{n}; \\
  1/R(x)  &\,=\, 1-\sum_{n\geq 1} \bigl(p(1)+p(2)+\cdots+p(n-1)-p(n) \bigr)x^n.
\end{align*}
See Equation~\eqref{R-def}, Proposition~\ref{R-prop} and
Proposition~\ref{exp-formula}.

\section{Factoring permutations with few inversions}

Let $\sym_n$ be the set of permutations on $[n]=\{1,2,\dots,n\}$. The
\emph{inversion table} of $\pi=a_1a_2\dots a_n$ in $\sym_n$ is defined as
$b_1b_2\dots b_n$ where $b_i$ is the number of elements to the left of
and larger than $a_i$; in other words, $b_i$ is the cardinality of the
set $\{ j\in[i-1] : a_j > a_i \}$. For instance, the inversion table of
$3152746$ is $0102021$. The number of inversions in $\pi$, denoted
$\inv(\pi)$, is simply the sum of the entries in the inversion table for
$\pi$.  We will work with an invertible transformation of the inversion
table that we call the \emph{cumulative inversion table}. It is obtained
by taking partial sums of the inversion table: $b_1$, $b_1+b_2$,
$b_1+b_2+b_3$, etc. The cumulative inversion table of $3152746$ is
$0113356$.

% \begin{lemma}
%   Inversion tables of $213$-avoiding permutations are increasing.
% \end{lemma}

A \emph{subdiagonal sequence} is a sequence of non-negative integers
whose $k$-th entry is smaller than $k$. It is easy to see that the
inversion table of a permutation is a subdiagonal sequence and that any
such sequence is an inversion table, so the two concepts
can be used interchangeably.

\begin{lemma}\label{inv-catalan}
  There are exactly $C_n=\binom{2n}{n}/(n+1)$ weakly increasing
  sub\-diagonal sequences of length $n$.
\end{lemma}
\begin{proof}
  Let a weakly increasing subdiagonal sequences $b_1b_2\dots b_n$ be
  given, and form the sequence $a_1a_2\dots a_n$ by setting
  $a_i=b_i+1$. Then $a_i \leq i$ and $1 \leq a_1 \leq a_2 \leq \dots \leq a_n$.
  By Exercise~6.19(s) in \cite{stanley} there are exactly
  $C_n$ such sequences.
\end{proof}

Let $\sym_n^k = \{\pi\in\sym_n : \inv(\pi) = k\}$ be the set of
permutations of $[n]$ with $k$ inversions, and let $\xperm_n$ be the
subset of $\sym_n^{n-1}$ consisting of those permutations whose every
prefix of length $k\ge1$ has fewer than $k$ inversions. For
$n=0,1,2,3,4$ those are $\{\eps\}$, $\{1\}$, $\{21\}$,
$\{231,312\}$, and $\{1432,2341,2413,3142,4123\}$, where $\eps$ is
the empty permutation.

\begin{lemma}
  For $n\geq 1$ we have $|\xperm_n| = C_{n-1}$.
\end{lemma}
\begin{proof}
  Clearly, the cumulative inversion table $\gamma=c_1c_2\dots c_n$ of any
  permutation $\pi\in\sym_n$ is weakly increasing. Also, the last letter,
  $c_n$, of $\gamma$ is the number of inversions in $\pi$. In particular, if
  $\pi\in\xperm_n$ then $c_n=n-1$ and $\pi$ is uniquely determined by
  $\gamma=c_1c_2\dots c_{n-1}$.  Now, any $k$-prefix of $\gamma$ is the
  cumulative inversion table of a permutation with fewer than $k$~inversions.
  Moreover, since the only condition on $\pi$ is that each $k$-prefix has
  fewer than~$k$ inversions, any weakly increasing subdiagonal sequence of
  length $n-1$ is the cumulative inversion table of such a permutation.  As
  pointed out in Lemma~\ref{inv-catalan}, such sequences are counted by the
  Catalan numbers.
\end{proof}

If $\alpha$ and $\beta$ are permutations, their \emph{direct sum},
denoted $\alpha\op\beta$, is the concatenation of $\alpha$ and $\beta'$,
where $\beta'$ is the transformation of $\beta$ that adds to each of its
letters the largest letter of $\alpha$.  Every permutation $\pi$ can be
written uniquely as the direct sum of its \emph{components}, which are
the minimal segments in a direct sum decomposition of $\pi$.  For
example, $23145867=231\op1\op1\op312$ has components 231, 1, 1, and 312. A
permutation consisting of a single component is \emph{indecomposable}.
Let $\comp(\pi)$ be the number of components in $\pi$. We will need a
lemma by Claesson, Jel{\'{\i}}nek and 
Steingr{\'{\i}}msson~\cite[Lemma~8]{cla-jel-est} relating $\comp(\pi)$ and
$\inv(\pi)$:

\begin{lemma}[\cite{cla-jel-est}]\label{CJS}
  For any permutation $\pi$,
  $$
  \inv(\pi) + \comp(\pi) \geq |\pi|.
  $$
\end{lemma}

\begin{lemma}\label{decomp}
  Let $d$ be a constant. If $\inv(\pi) \leq |\pi|+d$ and $\pi$ is
  decomposable, say $\pi=\alpha\oplus\beta$ with $\alpha$ indecomposable,
  then $\inv(\beta) \leq |\beta|+d+1$
\end{lemma}

\begin{proof}
  Since $\alpha$ is indecomposable we have
  $\inv(\alpha) \geq |\alpha|-1$ by Lemma~\ref{CJS}, and so
  \begin{align*}
    \inv(\beta)
    &\leq |\pi| + d - \inv(\alpha) \\
    &\leq |\pi| + d - (|\alpha| - 1) \\
    &\leq |\beta| + d + 1.\qedhere
  \end{align*}
\end{proof}

By iterated use of Lemma~\ref{decomp} we arrive at the following
generalisation of said lemma.

\begin{lemma}\label{last-comp}
  Let $d$ be a constant. If $\inv(\pi) \leq |\pi|+d$ and
  $\pi=\alpha_1\oplus\alpha_2\oplus\dots\oplus\alpha_m$ with each
  $\alpha_i$ indecomposable, then $\inv(\alpha_m) \leq |\alpha_m|+d+m-1$.
\end{lemma}

Recall now that $S_i(x)$ is the generating function for permutations of
length~$n$ with $n-i$ inversions:
$$ S_i(x) = \sum_{n\geq 0}|\sym_n^{n-i}|x^n.
$$
Also, let
$C(x) = (1-\sqrt{1-4x})/(2x)$
be the generating function for the Catalan numbers,
$C_n=\binom{2n}{n}/(n+1)$.

\begin{theorem}\label{S1-S0}
  We have
  $$
  \sym_n^{n-1} \simeq \bigcup_{k=0}^n \sym_{k}^{k}\times\xperm_{n-k}
  $$
  and thus the generating functions $S_0(x)$ and $S_1(x)$ satisfy the
  identity
  $$S_1(x) = xC(x)S_0(x).
  $$
\end{theorem}

\begin{proof}
  Let $\pi=a_1a_2\dots a_n\in\sym_n^{n-1}$. We shall ``factor'' $\pi$
  into two parts $\sigma$ and $\tau$ such that, for some $k$ in
  $\{0,1,\dots,n\}$, $\sigma$ belongs to $\sym_k^k$ and $\tau$ belongs
  to $\xperm_{n-k}$. Let $\sigma=a_1a_2\dots a_k$ be the longest prefix
  (possibly empty) of $\pi$ with as many letters as inversions and let
  $\tau=a_{k+1}a_{k+2}\dots a_n$ consist of the remaining letters of $\pi$.
  For instance, $\pi=4213675$ factors
  into $\sigma=4213$ and $\tau=675$.  By definition, $\inv(\sigma)=k$.
  We shall prove that $\sigma$ is a permutation of $[k]$, and thus
  $\tau$ is a permutation of $\{k+1,k+2,\dots,n\}$. Let
  $$d = \#\bigl\{\, (i,j) \,:\, a_i > a_j,\, i \leq k,\, j > k \,\bigl\}.
  $$
  That is, $d$ is the number of inversions in $\pi$ with one leg in
  $\sigma$ ($i\leq k$) and the other leg in $\tau$ ($j>k$). Then
  $\inv(\pi) = \inv(\sigma) + \inv(\tau) + d$.  We want to prove that $d=0$.
  Suppose to the contrary that $d\geq 1$. Now,
  $$\inv(\tau) = n-1-k-d = |\tau| - (d+1)
  $$
  and it follows from Lemma~\ref{CJS} that $\tau$ has at least $d+1$
  components; let us write
  $\tau = \alpha_1\oplus\alpha_2\oplus\dots\oplus\alpha_m$ with
  $m\geq d+1$. Using Lemma~\ref{last-comp} we find that
  $$
  \inv(\alpha_m) \leq |\alpha_m| - d + m  \leq |\alpha_m| - 1.
  $$
  Since $\alpha_m$ is indecomposable we also have
  $\inv(\alpha_m) \geq |\alpha_m|+1$ by Lemma~\ref{CJS} and thus
  $\inv(\alpha_m)=|\alpha_m|-1$. Let
  $\beta=\alpha_1\oplus\dots\oplus\alpha_{m-1}$. Note that no inversion
  can have one leg in $\sigma$ and the other leg in $\alpha_m$. That is,
  if $(i,j)$ is an inversion with $i\leq k$ and $j>k$ then
  $j < n-|\alpha_m|$. This is because any such inversion would
  necessarily be accompanied by $m-1$ other inversions---one for each of
  the components $\alpha_1$, $\alpha_2$, \dots,
  $\alpha_{m-1}$---contradicting $m\geq d+1$. Thus
  $$
  \inv(\sigma\beta) = n-1 - (|\alpha_m|-1) = |\sigma\beta|
  $$
  and we have found a prefix of $\pi$ that is longer than $\sigma$ with
  as many letters as inversions, which contradicts the definition of
  $\sigma$.

  Having proved that $\sigma\in\sym_k^k$ it immediately follows that
  $\inv(\tau)=n-k-1$. It remains to prove that $\tau$ has no 
  nonempty prefix with as many inversions as letters, but that is
  trivially true as $\sigma$ together with any such prefix would, again,
  be a longer prefix than $\sigma$, with as many inversions as letters.
\end{proof}

The proof above can be generalised to prove the following.

\begin{theorem}\label{Si-S0}
  For $i\geq 0$ we have
  $$
  \sym_n^{n-i-1} \simeq \bigcup_{k=0}^n \sym_{k}^{k-i}\times\xperm_{n-k}
  $$
  and thus the generating functions $S_{i+1}(x)$ and $S_{i}(x)$ satisfy the
  identity
  $$S_{i+1}(x) = xC(x)S_i(x),
  $$
  Equivalently,
  $$S_i(x)=(xC(x))^iS_0(x).
  $$
\end{theorem}

%[Note to self: Using $C=1+xC^2$ we can rewrite any power of $C$ as
%$p+qC$ where $p$ and $q$ are polynomials. Can we get something from
%viewing $S_i = x^iC^i S_0$ this way? I.e. writing $S_i =
%x^i(p_i+q_iC)S_0$. See Wolfdieter Lang, On polynomials related to powers of
%the generating function of Catalan's numbers, Fib.\ Quart.\ (2000).]

While the above theorems represent some progress in understanding
permutations with few inversions one crucial piece of the puzzle is
missing. Theorem~\ref{Si-S0} relates all the $S_i(x)$'s to $S_0(x)$, but
we need a formula for $S_0(x)$, which is what we shall offer in the next
section.

\section{A formula for $S_0(x)$}

Let us write $\lambda\vdash n$ to indicate that $\lambda$ is an integer
partition of $n$, and $\mu\vDash n$ to indicate that $\mu$ is
an integer composition of $n$. Further, let
\[
  \Par(x) = \prod_{k\geq 1}\frac{1}{1-x^k}\quad\text{and}\quad
  \Comp(x) = \frac{1-x}{1-2x}
\]
be the generating functions for integer partitions and compositions.
% Let us write $\lambda\vdash n$ to indicate that $\lambda$ is an integer
% partition of $n$ and define the sets $\Par_n=\{\lambda: \lambda\vdash n\}$
% and $\Par=\cup_{n\geq 0}\Par_n$. Further, let
% \[
%   \Par(x) = \sum_{n\geq 0}|\Par_n|x^n = \prod_{k\geq 1}\frac{1}{1-x^k}
% \]
% be the generating functions for the number of integer
% partitions. Similarly, we write $\mu\vDash n$ to indicate that $\mu$ is
% an integer composition of $n$ and define $\Comp_n=\{\mu: \mu\vDash n\}$,
% $\Comp=\cup_{n\geq 0}\Comp_n$, and
% \[
%   \Comp(x) = \sum_{n\geq 0}|\Comp_n|x^n = \frac{1-x}{1-2x}.
% \]
With $\Par_+(x) = \Par(x)-1$ denoting the generating
function for nonempty integer partitions we have
$$\Par(x)^{-1}
= \frac{1}{1+\Par_+(x)}
= \sum_{k\geq 0} (-1)^k(\Par_+(x))^k.
$$
Thus $\Par(x)^{-1}$ counts signed tuples of nonempty integer partitions,
where the sign of such a tuple $(\lambda^1,\dots,\lambda^k)$ is $(-1)^k$.
Define
\begin{align}\label{R-def}
\Pic(x)
  &= \Comp(x)\Par(x)^{-1} \\
  &= 1 + x^3 + 2x^4 + 5x^5 + 9x^6 + 19x^7 + 37x^8 + \cdots \nonumber
\end{align}
Then $\Pic(x)$ counts elements of the set
$$
\PIC_n \,=\,
\bigl\{\, (\lambda^1,\dots,\lambda^k;\mu):\,
  \lambda^i\vdash n_i,\,\mu\vDash m,\, n_1+n_2+\dots+n_k+m = n
\,\bigl\},
$$
where the sign of the tuple $(\lambda^1,\dots,\lambda^k;\mu)$ is
$(-1)^k$. Writing $(\lambda^1,\dots,\lambda^k;\mu)\vdash n$ when
$(\lambda^1,\dots,\lambda^k;\mu)$ is in $\PIC_n$ we then have, by
definition,
$$
\Pic(x)
= \sum_{n\geq 0}\left(\,\sum_{(\lambda^1,\dots,\lambda^k;\,\mu) \,\vdash\, n} (-1)^k\,\right)x^n.
$$

For illustration we list the elements of $\PIC_3$ below. Negative
elements are found in the left column and positive elements in
the right column:
$$\arraycolsep=1pt
\begin{array}{rcl}
  (1&;& 11)\\
  (1&;& 2)\\
  (1,1,1&;& \eps)\\
  (11&;& 1)\\
  (111&;& \eps)\\
  (21&;& \eps)\\
  (2&;& 1)\\
  (3&;& \eps)
\end{array}\qquad
\begin{array}{rcl}
  (\emptyset&;& 111)\\
  (\emptyset&;& 12)\\
  (\emptyset&;& 21)\\
  (\emptyset&;& 3)\\
  (1,1&;& 1)\\
  (1,11&;& \eps)\\
  (1,2&;& \eps)\\
  (11,1&;& \eps)\\
  (2,1&;& \eps)
\end{array}
$$
The sequence $1, 0, 0, 1, 2, 5, 9, 19, 37, 74, \dots$ of coefficients of
$R(x)$ is recorded in entry A178841 of the OEIS~\cite{oeis}. There it is
said to count the number of \emph{pure inverting compositions} of $n$;
see Propositions~2~and~3 in~\cite{lau-mas}.

We are now in position to state our main result regarding $S_0(x)$.

\begin{theorem}\label{S0-thm}
  We have $S_0(x) = R(xC(x))$, or, equivalently,
  $S_0\bigl(x(1-x)\bigr) = R(x)$, which, by Theorem~\ref{Si-S0}, implies that
  $S_i(x) = \bigl(xC(x)\bigr)^iR(xC(x))$.
\end{theorem}

Before proving this we need to better understand what combinatorial
structures $R(x)$ enumerates, so we shall define a sign-reversing
involution $\phi$ on $\PIC$ that singles out a positive subset
$\Fix(\phi)$ of $\PIC$ for which
$$
\Pic(x) = \sum_{n\geq 0}|\Fix(\phi)\cap\PIC_n|\,x^n.
$$
First, however, we define the auxiliary function
\[
  \cut : \{\mu: \mu\vDash n\} \to
  \bigcup_{i=0}^n \{\lambda: \lambda\vdash i\}\times \{\mu: \mu\vDash n-i\}
\]
by $\cut(\mu) = (\lambda, \mu')$ where $\mu = \lambda\mu'$ and $\lambda$ is
the longest prefix of $\mu$ that is weakly decreasing, and thus defines a
partition. For instance, $\cut(311212) = (311,212)$,
$\cut(21)=(21,\eps)$, $\cut(12)=(1,2)$ and $\cut(\eps) = (\eps,\eps)$.
Let $\lir(\mu)$ be the length of the longest strictly increasing
prefix (also called \underline{l}eftmost \underline{i}ncreasing \underline{r}un)
of $\mu$. For instance,
$\lir(121)=2$, $\lir(213)=\lir(1122)=1$ and $\lir(\eps)=0$.

\begin{lemma}\label{split}
  Let $\lambda$ be a nonempty partition and $\mu$ a composition such
  that $\lir(\mu)$ is even. Then $\lir(\lambda\mu)$ is odd. Moreover,
  if $a$ is the last element of $\lambda$, then
  \begin{align*}
    \cut(\lambda\mu) &=
      \begin{cases}
        (\lambda, \eps) & \text{if $\mu=\eps$ is empty} \\
        (\lambda, \mu)  &\text{if $(b,\mu')=\cut(\mu)$ and $a<b$;} \\
        (\lambda b, \mu') &\text{if $(b,\mu')=\cut(\mu)$ and $a \geq b$.}
      \end{cases}
  \end{align*}
\end{lemma}

Note that if $\mu$ is nonempty and $\lir(\mu)$ is even, then the first
element of $\mu$ must be smaller than the second, and hence the longest
weakly decreasing prefix of $\mu$ is a singleton (the first letter of
$\mu$). Thus the lemma above covers all cases. We now define the
promised involution $\phi$ on $\PIC_n$.

\begin{definition}\label{phi-def}
  Let $(\lambda^1,\dots,\lambda^k;\mu)\vdash n$. If $\lir(\mu)$ is even
  then
  \begin{align*}
    \phi(\lambda^1,\dots,\lambda^k;\mu) &=
    \begin{cases}
      (\emptyset; \mu) & \text{if $k=0$;} \\
      (\lambda^1,\dots,\lambda^{k-1};\lambda^k\mu) & \text{if $k>0$.} \\
    \end{cases}\\
    \intertext{If $\lir(\mu)$ is odd and $(\rho x,\mu')=\cut(\mu)$ then}
    \phi(\lambda^1,\dots,\lambda^k;\mu) &=
    \begin{cases}
      (\lambda^1,\dots,\lambda^k, \rho x;\mu') & \text{if $\lir(\mu')$ is even;} \\
      (\lambda^1,\dots,\lambda^k, \rho; x\mu') & \text{if $\lir(\mu')$ is odd.}
    \end{cases}
  \end{align*}
\end{definition}

The idea behind the map is that we can create an involution by moving a
partition $\lambda$ back and forth between being considered as part of
the list of partitions or as a prefix of our composition $\mu$. The
parity of $\lir(\mu)$ allows us to know if we, so to speak, have already
prepended a $\lambda$ or not; indeed $\lir(\lambda \mu)$ is odd if
$\lir(\mu)$ is even.

Let us look at a few cases illustrating Definition~\ref{phi-def}. A simple case
is that of a fixed point: $\lir(3644)=2$ is even and
\begin{align*}
  \phi(\emptyset;3644)  &= (\emptyset;3644).\\
  \intertext{
  Consider $(\lambda^1;\mu)=(6211;\eps) \vdash 10$. Then $\lir(\mu)=0$ is even, $k=1$ and
  }
  \phi(6211;\eps) &= (\emptyset;6211).\\
  \intertext{Another example of when $\lir(\mu)$ is even is}
  \phi(11,62;243352) &= (11;62243352).\\
  \intertext{Finally, two cases when $\lir(\mu)$ is odd are}
  \phi(11,62;643452) &= (11,62,643;452);\\
  \phi(11,62;643425) &= (11,62,64;3425).
\end{align*}

\begin{theorem}\label{fixed-points}
  The map $\phi$ is a sign-reversing involution on $\PIC_n$
  whose fixed points are of the form $(\emptyset; \mu)$ with $\mu\vDash n$
  and $\lir(\mu)$ even.
\end{theorem}

\begin{proof}
  Let $w =(\lambda^1,\dots,\lambda^k;\mu)\vdash n$ be given. It is clear
  that $\phi(w)\vdash n$ and that the first case of the definition,
  namely $\lir(\mu)$ is even and $k=0$, covers all fixed
  points. Further, the second case shortens the list of partitions by
  one while the third and fourth cases lengthen the same list by one. In
  all three cases the sign of $w$ is thus reversed. It remains to show that
  $\phi(\phi(w))=w$ and we consider each of the last three cases of the definition 
  of $\phi$ separately.

  If $\lir(\mu)$ is even and $k>0$, then
  $\phi(w)=(\lambda^1,\dots,\lambda^{k-1};\lambda^k\mu)$. To show that
  $\phi(\phi(w)) = w$ we consider the three cases of
  Lemma~\ref{split}.  If $\mu$ is empty then
  $\cut(\lambda^k\mu)=(\lambda^k,\eps)$, $\lir(\eps)=0$ is even and
  $$\phi(\lambda^1,\dots,\lambda^{k-1};\lambda^k\mu)
  = (\lambda^1,\dots,\lambda^{k-1},\lambda^k;\mu) = w.
  $$
  If $\mu$ is nonempty then let $(b,\mu')=\cut(\mu)$. Also, let $a$ be
  the last element of $\lambda^k$. If $a<b$ then
  $\cut(\lambda^k\mu) = (\lambda^k,\mu)$, $\lir(\mu)$ is even (by
  assumption) and
  $$\phi(\lambda^1,\dots,\lambda^{k-1};\lambda^k\mu)
  = (\lambda^1,\dots,\lambda^{k-1},\lambda^k;\mu) = w.
  $$
  If $a\geq b$ then $\cut(\lambda^k\mu) = (\lambda^kb,\mu')$,
  $\lir(\mu')=\lir(\mu)-1$ is odd and
  $$\phi(\lambda^1,\dots,\lambda^{k-1};\lambda^k\mu)
  = (\lambda^1,\dots,\lambda^{k-1},\lambda^k;b\mu')
  = (\lambda^1,\dots,\lambda^k;\mu) =  w.
  $$

  If $\lir(\mu)$ is odd then let $(\rho x,\mu')=\cut(\mu)$. If, in
  addition, $\lir(\mu')$ is even, then
  $$\phi(\phi(w))
  = \phi(\lambda^1,\dots,\lambda^k, \rho x;\mu')
  = (\lambda^1,\dots,\lambda^k; \rho x\mu') = w.
  $$
  If $\lir(\mu')$ is odd, then
  $\phi(w)=(\lambda^1,\dots,\lambda^k, \rho; x\mu')$ and, since
  $\lir(x\mu')$ is even,
  $$
  \phi(\phi(w))
  = \phi(\lambda^1,\dots,\lambda^k, \rho; x\mu')
  = (\lambda^1,\dots,\lambda^k; \rho x\mu') = w,
  $$
  which concludes the last case and thus also the proof.
\end{proof}

Next we aim at proving Theorem~\ref{S0-thm}. That is, we wish to prove
that
\begin{equation}\label{S0-eq}
  S_0(x(1 - x)) = R(x)
\end{equation}
and we start by giving a combinatorial interpretation of
$S_0(x(1 - x))$.

Let $T_{n,k}$ be the set of pairs $(S,\beta)$, where
$S\subseteq [n-k]$, $|S|=k$, and
\[
  \beta=(\beta_1,\beta_2,\dots,\beta_{n-k})
\] is a subdiagonal sequence with sum
$\beta_1 + \beta_2 +\cdots+ \beta_{n-k} = n-k$. Also, let
\[
  T_n = \bigcup_{k = 0}^n T_{n, k}.
\]
For instance, $T_0=\{(\emptyset, \eps)\}$, $T_1$ and $T_2$ are empty,
$T_3 = \{(\emptyset, 012)\}$,
and $T_4$ consists of the following $8$ elements:
\begin{gather*}
  (\emptyset, 0121),\, (\emptyset, 0112),\, (\emptyset, 0103),\, (\emptyset, 0022),\, (\emptyset, 0013),\\
  (\{1\},012),\, (\{2\},012),\, (\{3\},012).
\end{gather*}

\begin{lemma}\label{S0-lemma}
  We have
  \[
  S_0(x(1 - x)) = \sum_{n \geq 0} \Biggl(\sum_{(S,\beta)\,\in\, T_n} (-1)^{|S|}\Biggr)x^n.
  \]
\end{lemma}

\begin{proof}
  Let $(S,\beta)\in T_{n,k}$. View $\beta$ as an inversion table and let
  $\pi$ be the corresponding permutation on $[n-k]$. Note that $\pi$ has
  exactly $n-k$ inversions and thus the cardinality of $T_{n,k}$ is
  $\abs{\sym^{n-k}_{n-k}}\binom{n-k}{k}$. The result now follows from a
  direct calculation:\smallskip
  \begin{align*}
    S_0(x(1 - x))
    &= \sum_{n \geq 0} \abs{\sym^n_n} x^n (1 - x)^n \\
    &= \sum_{n \geq 0} \abs{\sym^n_n} x^n \sum_{k = 0}^n \binom{n}{k} (-1)^k x^k \\
    &= \sum_{n \geq 0} \Biggl(\,\sum_{k = 0}^n \abs{\sym^{n - k}_{n - k}} \binom{n - k}{k} (-1)^k\Biggr) x^n\\
    &= \sum_{n \geq 0} \sum_{k = 0}^n\Biggl(\sum_{(S,\beta)\,\in\, T_{n,k}} (-1)^{|S|}\Biggr)x^n \\
    &= \sum_{n \geq 0} \Biggl(\sum_{(S,\beta)\,\in\, T_n} (-1)^{|S|}\Biggr)x^n.
    \qedhere
  \end{align*}
\end{proof}

We shall show that the set $T_n$ in the inner summation in
Lemma~\ref{S0-lemma} can be replaced with a smaller set, but first a few
definitions.

For a composition $\mu=(\mu_1,\dots,\mu_k)$ define $\dmax(\mu)$ as $0$
if $k\leq 1$ and
\[
  \dmax(\mu) = \max\{\,\mu_j - j + 1: 2 \leq j \leq k \,\}
\]
otherwise. If we plot $\mu_j$ against $j$ this is the largest distance
it goes over the line $y=x-1$, excluding $\mu_1$ for technical reasons. For
instance, if $\mu=3241261$ then $\dmax(\mu)=\mu_3-3+1=2$ as depicted
below:
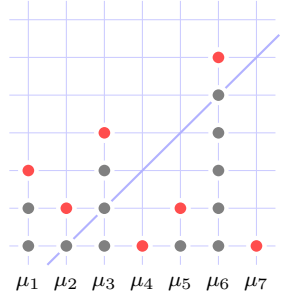
\[
  \begin{tikzpicture}[scale=0.5]
    \draw[blue!20,thin] (-0.5,0) -- (6.5,0);
    \draw[blue!20,thin] (-0.5,1) -- (6.5,1);
    \draw[blue!20,thin] (-0.5,2) -- (6.5,2);
    \draw[blue!20,thin] (-0.5,3) -- (6.5,3);
    \draw[blue!20,thin] (-0.5,4) -- (6.5,4);
    \draw[blue!20,thin] (-0.5,5) -- (6.5,5);
    \draw[blue!20,thin] (-0.5,6) -- (6.5,6);
    \draw[blue!20,thin] (0,-0.5) -- (0,6.5);
    \draw[blue!20,thin] (1,-0.5) -- (1,6.5);
    \draw[blue!20,thin] (2,-0.5) -- (2,6.5);
    \draw[blue!20,thin] (3,-0.5) -- (3,6.5);
    \draw[blue!20,thin] (4,-0.5) -- (4,6.5);
    \draw[blue!20,thin] (5,-0.5) -- (5,6.5);
    \draw[blue!20,thin] (6,-0.5) -- (6,6.5);
    \draw[thick, blue!30] (0.5,-0.5) -- (6.6,5.6);
    \foreach \x/\y in {0/0,0/1,0/2,1/0,1/1,2/0,2/1,2/2,2/3,3/0,4/0,4/1,5/0,5/1,5/2,5/3,5/4,5/5,6/0} {
      \draw[ultra thick, white, fill=gray] (\x,\y) circle (0.21);
    }
    \foreach \x/\y in {0/2,1/1,2/3,3/0,4/1,5/5,6/0} {
      \draw[ultra thick, white, fill=red!70] (\x,\y) circle (0.21);
    }
    \node at (0,-1) {$\scriptstyle{\mu_1}$};
    \node at (1,-1) {$\scriptstyle{\mu_2}$};
    \node at (2,-1) {$\scriptstyle{\mu_3}$};
    \node at (3,-1) {$\scriptstyle{\mu_4}$};
    \node at (4,-1) {$\scriptstyle{\mu_5}$};
    \node at (5,-1) {$\scriptstyle{\mu_6}$};
    \node at (6,-1) {$\scriptstyle{\mu_7}$};
  \end{tikzpicture}
\]
% \[
%   \begin{tikzpicture}[scale=0.45]
%     \foreach \x/\y in {0/0,0/1,0/2,1/0,1/1,2/0,2/1,2/2,2/3,3/0,4/0,4/1,5/0,5/1,5/2,5/3,5/4,5/5,6/0} {
%       \fill[gray] (\x+0.1,\y+0.1) rectangle +(0.8,0.8);
%       \fill[gray!20] (\x+0.15,\y+0.15) rectangle +(0.7,0.7);
%     }
%     \draw[semithick]
%           (0,0) -- (1,0) -- (1,1) -- (2,1) -- (2,2)
%        -- (3,2) -- (3,3) -- (4,3) -- (4,4) -- (5,4)
%        -- (5,5) -- (6,5) -- (6,6) -- (7,6);
%     \node at (0.55,-0.65) {$\scriptstyle{\mu_1}$};
%     \node at (1.55,-0.65) {$\scriptstyle{\mu_2}$};
%     \node at (2.55,-0.65) {$\scriptstyle{\mu_3}$};
%     \node at (3.55,-0.65) {$\scriptstyle{\mu_4}$};
%     \node at (4.55,-0.65) {$\scriptstyle{\mu_5}$};
%     \node at (5.55,-0.65) {$\scriptstyle{\mu_6}$};
%     \node at (6.55,-0.65) {$\scriptstyle{\mu_7}$};
%   \end{tikzpicture}
% \]

Up until this point we have listed the parts of a partition $\lambda$ in
weakly decreasing order. In what follows, it will be convenient to
instead list them in weakly increasing order. For instance, we may
write $\lambda = (\lambda_1,\lambda_2,\lambda_3) = (1,3,4)\vdash 8$.

\begin{definition}\label{good-pair}
  Let $\lambda$ be an integer partition and $\mu$ an integer composition.
  Let their total sum be $n$ and let $d=\dmax(\mu)$.
  % Let $\lambda \vdash i$, $\mu\vDash n-i$ and $d=\dmax(\mu)$.
  We shall write
  \[(\lambda, \mu)\Vdash n
  \]
  if the following three conditions hold:
  \begin{itemize}
  \item $\lambda$ has distinct parts (and is hence strictly increasing);
  \item $\lambda \neq \eps \implies \lambda_{\abs{\lambda}} < d$,
  \item $\mu\neq \eps \implies \mu_1 \leq d$.
  \end{itemize}
\end{definition}
For instance, $(\lambda,\mu)$ with $\mu=3241261$ as in the example above
does not satisfy Definition~\ref{good-pair} regardless of what the
partition $\lambda$ is; the reason being that $3=\mu_1 > \dmax(\mu) =
2$.  Let us consider the sets of pairs $(\lambda, \mu)\Vdash n$ for
small $n$. For $n=0$ there is a single pair, $(\eps, \eps)$;
for $n=1,2$ there are none; for $n=3$ there is a single pair,
$(\eps, 12)$; for $n=4$ there are two, $(\eps, 121)$ and
$(\eps, 13)$; and for $n=5$ there are seven:
\begin{gather*}
  (\eps, 113),\,
  (\eps, 1211),\,
  (\eps, 122),\,
  (\eps, 131),\,
  (\eps, 14),\,
  (\eps, 23),\,
  (1, 13).
\end{gather*}
As a larger example we offer $(134, 161121)\Vdash 20$.

\begin{theorem}\label{part1}
  We have
  \[
    \sum_{(S,\beta)\,\in\, T_n} (-1)^{|S|}
    = \sum_{(\lambda, \mu) \,\Vdash\, n} (-1)^{\abs{\lambda}}.
  \]
\end{theorem}

\begin{proof}
  We shall give a sign-reversing
  involution on $T_n$ whose fixed points can be bijectively mapped
  to pairs $(\lambda, \mu) \Vdash n$.

  Let $(S, \beta) \in T_{n,n-r}$ with
  $\beta = (\beta_1, \beta_2, \dots, \beta_r)$.  We will say that
  $\beta_i$ is \emph{marked} if $i \in S$. An index $i$ such that
  $\beta_i = 0$ and $\beta_{i+1} > 0$ will be called a \emph{0-ascent}.
  If $\beta_i = i - 1$, then we call $i$ a
  \emph{diagonal index} and $\beta_i$ a \emph{diagonal entry}.
  We shall now define an endofunction
  \[\psi : T_n \rightarrow T_n
  \]
  which we will later prove is a
  sign-reversing involution. Consider the entries $\beta_i$ in descending
  order by index and define $\psi(S, \beta)$ according to which of the
  following four mutually exclusive cases is encountered first:

  \begin{itemize}
  \item[1.] If $\beta_i$ is marked and there is no 0-ascent $j > i$,
    then we replace $\beta_i$ with an unmarked bigram $xy$ whose first
    letter is zero, $x=0$, and whose last letter is $y=\beta_i + 1$. In
    particular, $S$ is mapped to $S\setminus\{i\}$.\smallskip

  \item[2.] If $\beta_i$ is marked, $i$ is not a diagonal index and
    there is a 0-ascent $j > i$, then we replace $\beta_i$ with an
    unmarked $\beta_i + 1$ and append an unmarked zero to the end of
    $\beta$. Again, $S$ is mapped to
    $S\setminus\{i\}$.\smallskip

  \item[3.] If $\beta_i$ and $\beta_{i+1}$ are both unmarked, $i$ is a
    0-ascent and there is no diagonal index $j > i + 1$, then we replace
    the bigram $\beta_i\beta_{i+1}$ by a single marked
    $\beta_{i+1}-1$. Here, $S$ is mapped to $S\cup \{i\}$.\smallskip

  \item[4.] If $\beta_i \neq 0$ is unmarked, $\beta_r = 0$ and there is
    some 0-ascent $j > i$, then we replace $\beta_i$ by a marked
    $\beta_i - 1$ and remove $\beta_r$. Here, $S$ is mapped to
    $S\cup \{i\}$.
  \end{itemize}

  If none of these cases are encountered we let
  $\psi(S, \beta)=(S, \beta)$ be a fixed point. It is easy to see that
  each case preserves subdiagonality. Cases 1 and 2 remove a mark,
  increase the sum by one and add an element; consequently the image
  $\psi(S, \beta)$ is in $T_{n,n-r-1}$.  Cases 3 and 4 add a mark,
  decrease the sum by one and remove an element, so in these two cases
  $\psi(S, \beta)$ is in $T_{n,n-r+1}$. Thus $\psi$ is well-defined and
  sign-reversing. Let us consider some examples:

  \begin{itemize}
  \item Case~1 at $i = 3$: $\psi(\{1, 3\}, 0103) = (\{1\}, 01013)$
  \item Case~2 at $i = 3$: $\psi(\{2, 3\}, 0010150) = (\{2\}, 00201500)$
  \item Case~3 at $i = 4$: $\psi(\{1\}, 002040) = (\{1, 4\}, 00230)$
  \item Case~4 at $i = 2$: $\psi(\{3\}, 0120250000) = (\{2, 3\}, 002025000)$
  \item A fixed point: $\psi(\{1, 3\}, 0020152000) = (\{1, 3\}, 0020152000)$.
  \end{itemize}

  Next we shall prove that $\psi$ is an involution; that is,
  $\psi(\psi(S, \beta))=(S, \beta)$. If $(S, \beta)$ is a fixed point,
  then the claim is trivially true, so we can assume that we encounter
  one of the four cases above. Suppose $\psi(S, \beta) = (T, \gamma)$
  after falling into one of the cases at $\beta_i$. We want to show that
  $\psi(T, \gamma)=(S, \beta)$. The map $\psi$ leaves the suffix
  $\beta_{i+2}\beta_{i+3}\dots\beta_r$ of $\beta$ unchanged, aside from
  possibly appending an unmarked trailing zero; hence this suffix, with
  possibly an appended zero, will also be present in $\gamma$. If no
  zero was appended then clearly none of the cases apply to
  $(T, \gamma)$ at $j > i$, or else that case would have applied to
  $(S, \beta)$ as well. Suppose a zero was appended.  Since this
  trailing zero is unmarked $(T, \gamma)$ cannot fall into case~1 or 2
  for any $j > i$. Adding a zero at the end cannot introduce a 0-ascent,
  so $(T, \gamma)$ cannot fall into case~3 for $j > i$. Case~4 is also
  easy to exclude, so we conclude that $(T, \gamma)$ cannot fall into
  any of the four cases at an index $j > i$.

  Going through each of the cases at index $i$, we see that if $\beta_i$ falls into case~1,
  then $\gamma_i$ must fall into case~3, which undoes what case~1 just did. Similarly,
  case~2 is cancelled by case~4, 3 by 1, and 4 by 2; thus, $\psi$ is an involution.

  We now consider the fixed points of $\psi$. We wish to show that any
  nonempty fixed point $(S,\beta)$ of $\psi$, when considered as a marked
  sequence, can be written
  \[\sigma\tau\zeta
  \]
  where each letter of $\sigma$ is either a marked diagonal entry or an
  unmarked zero, ending in an unmarked zero;
  $\tau$ consists of unmarked positive entries at least one of
  which is a diagonal entry; and $\zeta$ is a (possibly empty) sequence
  of zeros.
  One instance of a fixed point is $(\emptyset, 0121)$ in
  which $\sigma=0$, $\tau=121$ and $\zeta=\epsilon$. Another
  instance is $(\{1, 3\}, 0020152000)$ in which $\sigma=0020$,
  $\tau=152$ and $\zeta=000$.

  Since $\beta$ is subdiagonal it starts with a zero. Moreover,
  its sum $r$ is positive, and hence
  it must contain a 0-ascent. Let $\sigma$ be the prefix of $\beta$
  consisting of every letter of $\beta$ up to and including
  the rightmost 0-ascent. Define $\tau$ as the subsequent contiguous run of positive
  entries in $\beta$ and let the remaining suffix be $\zeta$.
  In particular, $\tau$ is nonempty. Also, $\zeta$ must consist entirely of zeros;
  otherwise, it would contain a 0-ascent,
  contradicting that $\zeta$ is right of the rightmost 0-ascent in
  $\beta$. Now, if $\beta_i$ is marked
  and there is no 0-ascent $j > i$, then case~1 would apply at $\beta_i$.
  Thus every entry of $\sigma_{\abs{\sigma}}\tau\zeta$ must be unmarked.
  There also has to be a diagonal entry in $\tau$, otherwise
  the bigram $\sigma_{\abs{\sigma}} \tau_1$ would make us fall into case~3.
  Thus, there is an $\ell > 1$ such that $\tau_\ell = \ell - 1 + \abs{\sigma}$.
  If $\sigma_i$ is marked, then $i$ is a
  diagonal index, else we would fall into case~2 at $\sigma_i$ because it is to the left
  of a 0-ascent. To show that $\sigma$ is of the desired form we shall consider two cases.

  Suppose $\zeta$ is empty. Every element to the right of $\tau_\ell$
  is $\geq 1$, and $\tau_1 > 0$, so
  \[\tau_1 + \tau_\ell + \dots + \tau_{r - \abs{\sigma}}
    \geq 1 + (\ell - 1 + \abs{\sigma}) + (r - \abs{\sigma} - \ell) \geq r.
  \]
  The sum of entries in $\beta$ is $r$ (by definition of $T_{n,n-r}$) and
  consequently the sum above is exactly $r$. Thus,
  every entry of $\sigma$ is zero. Aside from the first one,
  none of those zeros can be marked, or else we would have a contradiction with the
  earlier result that any marked element of $\sigma$ is a diagonal entry. Thus
  $\sigma$ is of the desired form.

  Suppose $\zeta$ is nonempty. There cannot be any positive unmarked $\sigma_i$
  since then we would fall into case~4 at $\sigma_i$. Thus, $\sigma$ is
  of the desired form by the same argument as above.

  Let us now define a function $\theta$ mapping fixed points of $\psi$
  to pairs $(\lambda, \mu) \Vdash n$.
  Given a fixed point $(S, \beta)$ factored as
  $\sigma\tau\zeta$ we let $\theta(S, \beta)=(\lambda,\mu)$, where
  $\lambda$ consists of the marked indices of $\sigma$ written in increasing order
  and $\mu = \tau$. In other words, the entries of $\lambda$ are the elements of
  $S$, the reason being that $\tau$ and $\zeta$ contain only unmarked elements.
  For example $\theta(\emptyset, 0121) = (\epsilon, 121)$ and $\theta(\{1, 3\},
  0020152000) = (13, 152)$. It is clear that
  $\lambda$ has distinct parts and that $\mu$ defines a composition.
  Furthermore, the sum of values in $\lambda$ and $\mu$ is the sum
  of elements in $\beta$ plus the number of marked elements, which is $r + n - r = n$.
  Note that the sign simply is $(-1)^{\abs{\lambda}}$.

  We wish to show that $(\lambda, \mu) \Vdash n$. Our diagonal index
  $\ell$ gives us
  \begin{align*}
    \dmax(\mu)
    &\geq \mu_{\ell}-\ell+1 \\
    &= \tau_{\ell}-\ell+1 \\
    &= (\ell - 1 + \abs{\sigma}) - \ell + 1 = \abs{\sigma}.
  \end{align*}
  Suppose $\dmax(\mu) = \mu_j - j + 1$. Then
  \begin{align*}
   \dmax(\mu)
    &= \beta_{j+\abs{\sigma}} - j + 1 \\
    &\leq j + \abs{\sigma} - 1 - j + 1 = \abs{\sigma}.
  \end{align*}
  in which the inequality is a consequence of subdiagonality.
  Thus $\dmax(\mu) = \abs{\sigma}$. If $\lambda$ is nonempty, then $\lambda_1$
  corresponds to a marked diagonal index, which must then be in $\sigma$. Thus
  $\lambda < \abs{\sigma} = \dmax(\mu)$ since $\sigma$ ends on a zero. If
  $\mu$ is nonempty, then
  $\mu_1 = \tau_1 \leq \abs{\sigma}$ by
  subdiagonality, and hence $\mu_1 \leq \dmax(\mu)$.
  Thus $(\lambda, \mu) \Vdash n$.

  To complete our proof we have to show that $\theta$ is bijective,
  which we do by constructing its inverse. Assume that
  $(\lambda, \mu) \Vdash n$ and $k = \dmax(\mu)$. Let
  $\sigma=\sigma_1\sigma_2\dots\sigma_k$, where $\sigma_i=i-1$ is a
  marked diagonal entry if $i=\lambda_j$ for some $j\in [|\lambda|]$, and
  $\sigma_i=0$ is an unmarked zero otherwise. Also,
  let $\tau = \mu$ and let $\zeta$ be a segment consisting of $n - \abs{\sigma} - \abs{\tau}$
  unmarked zeros. By the same argument as above we
  have $\abs{\lambda}$ marked elements, and the sum of all elements
  in $\lambda$ and $\mu$ is $n - \abs{\lambda}$, so $r = \abs{\lambda}$.
  Since $\lambda \neq \epsilon \Rightarrow \lambda_1 < \dmax(\mu)$ we have
  $\sigma_{\abs{\sigma}} = 0$. Thus the image is in $T_{n, n - r}$ and
  $\theta$ maps $\sigma\tau\zeta$ back to $(\lambda, \mu)$, completing our proof.
\end{proof}

\begin{theorem}\label{part2}
  We have
  \[
    \sum_{(\lambda, \mu) \,\Vdash\, n} (-1)^{\abs{\lambda}} \,=\,
    \bigl|\{\mu\vDash n:\text{$\lir(\mu)$ even}\}\bigr|.
  \]
\end{theorem}

\begin{proof}
  Let $\sim$ be the equivalence relation generated by postulating that
  \[(\lambda a, \mu) \sim (\lambda, a \mu)
  \]
  whenever both $(\lambda a, \mu)\Vdash n$ and
  $(\lambda, a \mu)\Vdash n$ hold.  For example, when $n = 5$ the
  equivalence classes are all singletons except the class
  $\{(\epsilon, 113), (1, 13)\}$. For $n = 6$ there are three
  non-singleton classes, namely
  \[\{(\epsilon, 1131), (1, 131)\},\,
    \{(\epsilon, 114), (1, 14)\}\,\text{ and }\{(\epsilon, 123), (1, 23)\}.
  \]

  We wish to show that the inner sum in the expression for $R(x)$ above
  when restricted to a single equivalence class is $0$ or $1$.  In other
  words, if $C$ is an equivalence class, then
  \[
    \sum_{(\lambda, \mu)\,\in\, C} (-1)^{\abs{\lambda}} \in  \{0,1\}.
  \]
  Assume $(\lambda a, \mu)\Vdash n$. Then $a < \dmax(\mu)$, but
  $\dmax(a\mu) \geq \dmax(\mu) - 1$, so $a \leq \dmax(a\mu)$.
  Furthermore, if $\lambda$ is nonempty, then $\lambda_{\abs{\lambda}} < a$
  because $\lambda a$ is strictly increasing. Thus $\lambda_{\abs{\lambda}}
  < a \leq \dmax(a\mu)$, and so $(\lambda, a\mu)\Vdash n$. By induction
  on the number of elements moved
  we see that $(\lambda, \mu)$ is in the same equivalence class as
  $(\epsilon, \lambda\mu)$. Clearly we cannot have two pairs of the form
  $(\epsilon, \mu)$ in the same equivalence class, so we make them
  our representatives.

  Let $(\epsilon, \mu)$ be such a representative. We will call $k$ \emph{valid} if
  \[(\mu_1\dots\mu_k,\, \mu_{k+1}\dots\mu_{\abs{\mu}}) \,\Vdash\, n.
  \]
  Let $\lambda=\mu_1\dots\mu_k$ and $\nu=\mu_{k+1}\dots\mu_{\abs{\mu}}$.
  By the argument above, if some $k$ is valid, then all smaller $k$ are
  valid too. We want to find the largest valid $k$.  The sign of
  $(\mu_1\dots\mu_k, \mu_{k+1}\dots\mu_{\abs{\mu}})$ is $(-1)^k$, so if
  the largest valid value is $\ell$, then the sum of the equivalence
  class of $(\epsilon, \mu)$ is $(-1)^0 + (-1)^1 + \dots + (-1)^\ell$,
  which is zero if $\ell$ is odd and $1$ if $\ell$ is even.

  Let $s = \lir(\mu)$. We wish
  to show that $\ell$ and $s$ have the same parity. Clearly, $\ell \leq s$ since
  otherwise $\lambda=\mu_1\dots\mu_{\ell}$ would not be strictly increasing. If $\ell = s$ we have nothing left
  to prove, so we can assume that $\ell < s$. Then $\nu$ is nonempty and
  $\nu_1 \leq \dmax(\nu)$, so $\abs{\nu} \geq 2$ and $\ell \leq s - 2$. Let $k = s - 2$.
  Then $\lambda=\mu_1\dots\mu_k$ is strictly increasing and $\nu_1 < \nu_2$. Thus $\dmax(\nu)
  \geq \nu_2 - 1 \geq \nu_1$. If $\lambda$ is nonempty, then $\lambda_k <  \nu_1 \leq \dmax(\nu)$.
  Thus $k = s - 2$ is valid, so $\ell = s - 2$, which has
  the same parity as $s$.
  The representatives $(\epsilon, \mu)$ whose equivalence classes have sum one are
  hence exactly those where $\lir(\mu)$ is even.
\end{proof}

Theorem~\ref{S0-thm} follows directly from Lemma~\ref{S0-lemma}
and Theorems~\ref{fixed-points},~\ref{part1}~and~\ref{part2}.

\section{Structure of $R(x)$}

By Theorem~\ref{fixed-points} the elements of $\Fix(\phi)\cap\PIC_n$ are
of the form $(\emptyset; \mu)$ with $\mu\vDash n$ and $\lir(\mu)$
even. With this in mind let
\[\Fix_n(\phi) = \{\mu\vDash n:\text{$\lir(\mu)$ even}\}.
\]
Let $\M_n$ be the set of compositions of $n$ that start with an ascent
and are weakly decreasing after the initial ascent and let $M(x)$ be the
corresponding generating function. That is,
$(\mu_1, \dots, \mu_k) \in \M_n$ if and only if $k\geq 2$,
$\mu_1<\mu_2\geq\mu_3\geq\dots\geq\mu_k$ and
$\mu_1 + \dots + \mu_k = n$. For instance, $\M_n=\emptyset$ for
$n\leq 2$, $\M_3=\{12\}$, $\M_4=\{121,13\}$,
$\M_5=\{1211,122,131,14,23\}$ and the first few terms of the
power series $M(x)$ are
\[
  M(x) =
  x^3 + 2x^4 + 5x^5 + 8x^6 + 15x^7 + 23x^8 + 37x^9 + \cdots
\]
Let $\mu^1, \mu^2, \dots, \mu^k$ be compositions with
$\mu^i\in\M_{n_i}$. Their concatenation
\[
  \mu=\mu^1\mu^2\cdots\mu^k
\]
is a composition of $n=n_1+\cdots+n_k$ with $\lir(\mu)$ even, and so
$\mu\in\Fix_n(\phi)$.

Conversely, given a composition $\mu\in\Fix_n(\phi)$, let $\mu^1$ be
the longest prefix  of $\mu$ that belongs to $\M_{n_1}$, where $n_1$ is the
length of $\mu^1$. Writing $\mu=\mu^1\nu$ we can recursively do the
same with $\nu$, stopping if $\nu$ is empty. This way we arrive at a
factorisation $\mu=\mu^1\mu^2\cdots\mu^k$ with $\mu^i\in\M_{n_i}$ and
$n=n_1+\cdots+n_k$. For instance, the factors of $123511211\in \Fix_{17}(\phi)$
are $12$, $351$ and $1211$.

In terms of generating functions the factorisation we have established
translates to the functional equation $R(x) = (1 - M(x))^{-1}$.
Now, by \eqref{R-def},
\begin{align*}
  M(x) &= 1 + \biggl(\frac{x}{1-x} - 1\biggr)\Par(x).
\end{align*}
Thus, aside from the constant term, the coefficient of $x^n$ in $M(x)$ equals
\begin{equation}\label{M-coeffs}
  p(0)+p(1)+\dots+p(n-1)-p(n)
\end{equation}
and coincides with sequence A058884 in the OEIS~\cite{oeis}. Moreover,
\eqref{M-coeffs} is the number of compositions with exactly one
inversion according to Theorem~4.1 of~\cite{heubach-al-inv-comp}.
To summarise we have established the following proposition.
\begin{proposition}\label{R-prop}
  With $p(n)$ denoting the number of partitions of $n$,
  \[
    R(x) = \Bigl(1-\sum_{n\geq 1} \bigl(p(1)+p(2)+\cdots+p(n-1)-p(n) \bigr)x^n \Bigr)^{-1}.
  \]
\end{proposition}

An alternative formula can be obtained from considering the logarithmic
derivative of $R(x)$:

\begin{proposition}\label{exp-formula}
  With $\sumdiv(n)$ denoting the sum of the divisors of $n$,
  \[
    R(x) = \exp\left(\sum_{n\geq 1}\bigl(2^n-\sumdiv(n)-1\bigr)\, \frac{x^n}{n}\right).
  \]
\end{proposition}
\begin{proof}
  Taking the logarithmic derivative of \eqref{R-def} we get
  \begin{align*}
    x\bigl(\log R(x)\bigr)'
    &= \frac{x\Comp'(x)}{\Comp(x)} - \frac{x\Par'(x)}{\Par(x)} \\
    &= \frac{x}{(1-x)(1-2x)} - \sum_{k \geq 1} \frac{kx^k}{1 - x^k}
  \end{align*}
  An expression of the form $F(x)=\sum_{k \geq 1}a_kx^k/(1 - x^k)$ is called
  a Lambert series, and it is well known, and easy to see, that
  \[F(x)=\sum_{n \geq 1}b_nx^n,\,\text{ where }\, b_n=\sum_{k|n}a_k.
  \]
  In our case $a_k=k$ and hence $b_n= \sumdiv(n)$. Consequently,
  \begin{equation}\label{logarithmic-pointing}
    x\bigl(\log R(x)\bigr)' = \sum_{n \geq 1} (2^n - 1 - \sumdiv(n)) x^n
  \end{equation}
  and it follows that
  \begin{align*}
    \log R(x)
    &= \int_0^x \sum_{n \geq 1} (2^n - 1 - \sumdiv(n)) t^{n-1} dt \\
    &= \sum_{n\geq 1}\bigl(2^n-1-\sumdiv(n)\bigr)\frac{x^n}{n},
  \end{align*}
  which proves the claim.
\end{proof}

\begin{corollary}
  The cardinalities $r_n=\abs{\PIC_n}$ can be computed recursively by
  $r_0=1$ and, for $n\geq 1$,
  \[r_n = \frac{1}{n} \sum_{k = 1}^n r_{n-k} (2^k - \sumdiv(k) - 1).
  \]
  Moreover, we have the closed formula
  \[r_n =
    \frac{1}{n!} \sum_{\pi \in \Sym(n)} \prod_{\ell \in C(\pi)} (2^\ell - \sumdiv(\ell) - 1),
  \]
  where $\Sym(n)$ is the symmetric group of degree $n$ and $C(\pi)$ is a multiset that
  encodes the cycle type of $\pi$, that is, there is an $\ell \in C(\pi)$ for
each $\ell$-cycle of $\pi$.
\end{corollary}
\begin{proof}
  Since $\bigl(\log R(x)\bigr)' = R'(x)/R(x)$ it follows from
  \eqref{logarithmic-pointing} that
  \[xR'(x) = R(x)\sum_{n \geq 1} (2^n - 1 - \sumdiv(n)) x^n
  \]
  and on identifying coefficients we get the claimed recursion. For the
  closed formula we refer to Equation~(8) in \cite{scoreseq} and the
  paragraph preceding that formula.
\end{proof}

It easy to see that the coefficient of $x^n$ in $xM'(x)/(1 - M(x))$ is
$2^n - \sumdiv(n) - 1$. Thus, if we consider the factorisation
$\mu=\mu^1\mu^2\cdots\mu^k$ of elements in $\PIC$, as above, together
with a distinguished site of $\mu^1$, then such structures should be
counted by $2^n - \sumdiv(n) - 1$. Finding a bijective proof of this
remains an open problem.

% \nocite{CATalg} % Why???

\bibliographystyle{abbrv}
\bibliography{few-inversions.bib}

\end{document}